\documentclass[twoside,reqno]{amsart}
\usepackage[english]{babel}
\usepackage[latin1]{inputenc}
\usepackage[T1]{fontenc}
\usepackage{xcolor}
\definecolor{codegreen}{rgb}{0,0.6,0}
\usepackage{amssymb,amsmath,amstext,amsfonts,amsthm,mathrsfs}
\usepackage{mathtools}
\usepackage{relsize}
\usepackage[colorlinks=true, urlcolor=blue, linkcolor=blue, citecolor=blue]{hyperref}
\usepackage{todonotes}
\usepackage[normalem]{ulem}
\usepackage{graphicx}
\usepackage{float}
\usepackage{enumitem}
\usepackage{multirow}
\usepackage{booktabs}
\usepackage{diagbox}
\usepackage[table,xcdraw]{xcolor}
\usepackage{listings}
\lstdefinelanguage{Macaulay2}
{
xleftmargin=.2in, 
xrightmargin=.2in, 
basicstyle={\ttfamily}, 
keywordstyle={\color{blue}}, 
commentstyle={\color{codegreen}}, 
stringstyle={\color{red!40!black}}, 
rulecolor=\color{yellow}, 
basewidth={1.2ex}, 
sensitive=false, 
morecomment=[l]{--}, 
morecomment=[s]{-*}{*-}, 
morestring=[b]", 
escapechar={`}, 
escapebegin={\rmfamily}, 
morekeywords={load, random, degree, genus, topComponents, ideal, Ext, minors, quotient, intersect, map, kernel, preimage, codim, sheaf, matrix, hilbertPolynomial, Projective, false, sheafExt, ann, cooker, flatten, gens, entries, basis, apply}
}
\lstalias{Macaulay2output}{Macaulay2}

\usepackage{tikz-cd}
\usetikzlibrary{arrows}
\tikzcdset{
every label/.append style = {font = \small},
arrow style=tikz,
}

\usepackage[all]{xy}

\theoremstyle{plain}

\newtheorem{theorem}{Theorem}[section]

\newtheorem{corollary}[theorem]{Corollary}
\newtheorem{proposition}[theorem]{Proposition}
\newtheorem{remark}[theorem]{Remark}
\newtheorem{example}[theorem]{Example}
\newtheorem{conjecture}[theorem]{Conjecture}

\newcommand{\ba}{{\bf a}}
\newcommand{\bb}{{\bf b}}
\newcommand{\lra}{\longrightarrow}
\newcommand{\ZZ}{\mathbb{Z}}
\newcommand{\op}[1]{{\mathcal O}_{\mathbb{P}^{#1}}}
\newcommand{\p}[1]{\mathbb{P}^{#1}}

\DeclareMathOperator{\im}{im}

\DeclareMathOperator{\End}{End}
\DeclareMathOperator{\gl}{GL}
\DeclareMathOperator{\Ext}{Ext}
\newcommand{\E}{\mathcal{E}}

\newcommand{\cA}{\mathcal{A}}
\newcommand{\cB}{\mathcal{B}}
\newcommand{\cC}{\mathcal{C}}
\newcommand{\cH}{\mathcal{H}}
\newcommand{\cM}{\mathcal{M}}

\DeclareMathOperator{\id}{id}
\newcommand{\OO}{\mathcal{O}}

\title[New irreducible components of $\mathcal{B}(0,c_2)$]{New irreducible components of $\mathcal{B}(0,c_2)$ 
and Computation of the Dimension of its tangent space}

\author[A. Fontes]{Aislan Fontes}
\address{Departamento de Matem\'atica \\ Universidade Federal de Sergipe - Campus Itabaiana\\ Av. Vereador Ol\'impio Grande s/n \\
Itabaiana - SE 49506-036, Brasil}
\email{ailfontes30@gmail.com}

\author[M. Santos]{Maxwell Santos}
\address{Departamento de Matem\'atica \\ Universidade Federal de Sergipe - Campus Itabaiana\\ Av. Vereador Ol\'impio Grande s/n \\
Itabaiana - SE 49506-036, Brasil}
\email{06maxwell06@gmail.com}

\makeatletter
\@namedef{subjclassname@2020}{\textup{2020} Mathematics Subject Classification}
\makeatother


\date{February 2026}

\begin{document}

\begin{abstract}
We provide a Macaulay2 code for computing the dimension of the tangent space to $\mathcal{B}(e,c_2)$ in certain cases. Using this 
code, we identify components of $\mathcal{B}(e,c_2)$ containing singular points and compute the dimension 
of the irreducible component $M_4$ of $\mathcal{B}(-1,6)$, whose existence was proved in 
\cite{MF2021}. Furthermore, we prove the existence of infinite families of 
irreducible components of $\mathcal{B}(0,c_2)$.
\end{abstract}

\maketitle
\tableofcontents

\section{Introduction}

 Let $\mathcal{B}(e,c_2)$ denote the moduli space of stable rank 2 bundle on $\mathbb{P}^3$ with Chern classes $c_1=e$ and $c_2$. Up to normalizations, it is sufficient to consider $e=0, -1$ and $c_2\geq0$. In \cite{H78}, Hartshorne showed that  $\mathcal{B}(0,c_2)$ is a quasiprojective scheme and nonempty for $c_2\geq1$ while $\mathcal{B}(-1,c_2)$ is a quasiprojective scheme and nonempty for all even integer $c_2\geq2$, where each irreducible component of $\mathcal{B}(e,c_2)$ has dimension at least $8c_2-3+2e$. It is worth recalling that, for $c_2\geq2$, three types of irreducible components of the moduli space $\mathcal{B}(e,c_2)$ are known: 
  \begin{itemize}
      \item An infinite series $I_{c_2}$ of irreducible component of expected dimension $8c_2-3+2e$ c.f. \cite{H78}. If $e=0$ then  $I_{c_2}$ is called \textit{Instanton component} which is the closure 
of the smooth open subset of $I_{c_2}$ constituted by the Instanton vector bundles;
\item Irreducible components of the type $\mathcal{G}(t,r,s)$, so-called \textit{Ein components}, depending on the triples of 
integers $t, r, s$ constructed by Ein \cite{Ein88}.
\item A new infinite series $\sum_e$ of irreducible components $\mathcal{M}_{c_2}$ of 
$\mathcal{B}(e,c_2)$, recently provided in \cite{tikhomirov2019}, which we call the \textit{modified Instanton component}.
  \end{itemize} 
  
  The existence of all irreducible components $\mathcal{B}(e,c_2)$ is known only for small values of $c_2$, specifically 
  $0\leq c_2\leq5$. More precisely, Hartshorne and Rao \cite[Section 5.3]{HR91} proved that the moduli spaces $\mathcal{B}(0,1)$ and 
  $\mathcal{B}(0,2)$ are irreducible. For $\mathcal{B}(0,3)$ and $\mathcal{B}(0,4)$, it is shown in \cite{chang83,Eli81} that each has two irreducible components. The space $\mathcal{B}(0,5)$ is characterized in \cite{almeida21}. Finally, classifications of $\mathcal{B}(-1,2)$ and $\mathcal{B}(-1,4)$ can be found in \cite{HS81,M81} and \cite{BM85}, respectively. Morevorer, from \cite[Theorem 18]{MF2021} it follows that $\mathcal{B}(-1.6)$ has four known irreducible components: the Hartshorne component $M_1$ of dimension 43, two Ein components $M_2$ and $M_3$ of dimension 43 and 50, respectively, and a fouth component $M_4$ of dimension at least 45.  \\
  \indent We organize this paper as follows. In Section \ref{section2} we recall the concept of minimal monads and its relation with stable rank 2 vector bundles on $\mathbb{P}^3$. We also list all the known irreducible components of $\mathcal{B}(0,c_2)$  with their respective dimension. Finally, we recall the Hartshorne-Serre correspondence, which constructs rank-2 vector bundles on $\mathbb{P}^3$.\\
 \indent Section \ref{section-ext} is devoted to computing the dimension of $\Ext^1(\mathcal{E}, \mathcal{E})$ for a stable rank 2 bundle $\mathcal{E}$ on $\mathbb{P}^3$ under certain specified conditions. To aid in this computation, we provide a Macaulay2 code. Using this method, we collect in Table \ref{ext-tab} the dimension of the tangent space $\mathcal{B}(0,6)$ at a stable bundle $\E$ given as cohomology of a monad corresponding to one of the cases 6 (3) or 6(4) in Table 5.3 of Hartshorne and Rao \cite{HR91}. We also prove in Theorem \ref{ex2} that the dimension of the irreducible component $M_4$ of $\mathcal{B}(-1,6)$ is 45.\\
 \indent Using the lower semi-continuity of cohomology dimensions for coherent sheaves and comparing the dimensions 
 of the three known irreducible components of $\mathcal{B}(0,c_2)$, we prove in Section \ref{section-families} 
 the existence of two new infinite families $\mathbf{V}_0$ and $\mathbf{V}_1$ of components of $\mathcal{B}(0,c_2)$. The family $\mathbf{V}_0$ contains stable rank-2 bundles arising as cohomology of minimal monads as in \eqref{eq:hoorcks-monad} with $\boldsymbol{a}=({(a-1)}^2,0^2), \boldsymbol{b}=(a^2), a\geq3$ while $\mathbf{V}_1$ contains those given as cohomology of minimal monads as in \eqref{eq:hoorcks-monad} with $\boldsymbol{a}=({(a-1)}^2,1)$ 
 and $\boldsymbol{b}=(a^2), a\geq3$. 
 
\section{Preliminaries}\label{section2}
This section will be devoted to recall basic concepts and prove some technical results used throughout the work. 

\subsection{Monads}\label{sub-monads}
A monad is a complex of vector bundles of the form 
\[
\cM : 0 \lra \cC \stackrel{\beta}{\lra} \cB \stackrel{\alpha}{\lra} \cA \lra 0
\]
whose only nontrivial cohomology is $\E = \cH^0(\cM) = \ker \alpha/\im \beta$. In particular, $\alpha$ is surjective and $\beta$ is injective. The sheaf $\E$ is the cohomology of the monad $\cM$. In \cite{Horrocks-constr}, Horrocks proved that any vector bundle $\E$ on $\p3$ can be obtained as the cohomology of a monad where $\cA$, $\cB$, and $\cC$ are direct sums of line bundles; see also \cite{Barth-Hulek-monads,Rao-coh}. Such monads are called Horrocks monads. In this case, $\alpha$ and $\beta$ are given by matrices of homogeneous polynomials. The monad is called minimal if the matrices defining $\alpha$ and $\beta$ have no non-zero constant entry. Any monad is homotopy equivalent to a minimal one.

For $\E$ a rank-two vector bundle with $c_1(\E)=0$, the isomorphism $\E \cong \E^\vee$ reflects on the monad: $\cC = \cA^\vee$ and $\cB = \cB^\vee$ and $\E$ is the cohomology of a minimal monad of the form
\begin{equation}\label{eq:hoorcks-monad}
    \cM : 0 \lra \bigoplus_{i=1}^s \op3(-a_i) \stackrel{\alpha}{\lra} \begin{matrix} \displaystyle
    \bigoplus_{i=1}^{s+1} \op3(b_i) \\ \displaystyle \oplus \\ \displaystyle \bigoplus_{i=1}^{s+1} \op3(-b_i)
\end{matrix} \stackrel{\beta}{\lra} \bigoplus_{i=1}^s \op3(a_i) \lra 0
\end{equation}
for some integers $a_1\leq \dots \leq a_s$ and $0 \leq b_1\leq \dots \leq b_{s+1}$. The tuples $\ba = (a_1, \dots , a_s)$ and $\bb = (b_1, \dots , b_{s+1})$ are called the type of the monad. Computing Chern classes yields the following equation:
\begin{equation}\label{eq:monad-c2}
c_2(\E)=\sum_{i=1}^sa_i^2-\sum_{j=1}^{s+1} b_j^2.
\end{equation}
\begin{example}\label{ein-monad}
\textup{A Horrocks monad such that the extreme terms are line bundles, $\ba = (t)$, is called an 
Ein monad. This terminology originates from the influential work of Ein \cite{Ein88}. For integers $t, r,s \geq 0$  
such that $r+s<t$, we consider monads of the form
\[
0 \lra \op3(-t) \stackrel{\alpha}{\lra} \begin{matrix}
  \displaystyle  \op3(s)\oplus \op3(r) \\
  \displaystyle  \oplus \\
  \displaystyle  \op3(-s) \oplus \op3(-r) \\
\end{matrix} 
\stackrel{\beta}{\lra} \op3(t) \lra 0
\]
where there exist polynomials $f_1\in H^0(\mathbb{P}^3,\mathcal{O}_{\mathbb{P}^3}(t-s)), f_2\in H^0(\mathbb{P}^3,\mathcal{O}_{\mathbb{P}^3}(t-r)), f_3\in H^0(\mathbb{P}^3,\mathcal{O}_{\mathbb{P}^3}(t+s))$ and $f_4\in H^0(\mathbb{P}^3,\mathcal{O}_{\mathbb{P}^3}(t+r))$ such that  
\[
\alpha = \begin{bmatrix}
    f_1 & f_2 & f_3 & f_4
\end{bmatrix} \quad \text{and} \quad \beta = \begin{bmatrix}
    -f_3 & -f_4 & f_1 & f_2
\end{bmatrix}^T.
\]
If $\E$ is the cohomology of a minimal monad as above, then $\E$ is called a \textit{generalized null correlation bundle}. It 
was shown in \cite{Ein88} that $\E$ is stable rank 2 bundle if and only if $t>r+s$ and $c_2(\E)=t^2-r^2-s^2$.
}
\end{example}
Given a minimal monad as in \eqref{eq:hoorcks-monad}, $\alpha$ determines a minimal presentation of the cohomology module $H^1_*(\E) = \bigoplus_{l\in \ZZ} H^1(\E(l))$ as a $S= k[x_0,\dots,x_3]$-module. Conversely, given a minimal presentation
\[
L_1 \stackrel{\alpha}{\lra} L_0 \lra H^1_*(\E) \lra 0
\]
we have that $\E$ is defined by the monad
\[
0 \lra \widetilde{L}_0^\vee(-1) \lra \widetilde{L}_1  \stackrel{\alpha}{\lra} \widetilde{L}_0 \lra 0. 
\]
For example, if $\E$ is a null correlation bundle as in Example \ref{ein-monad} then 
$H^1_*(\E)\simeq M\otimes S(t)$, where $M=S/(f_1,f_2,f_3,f_4)$ is a $0$-dimensional graded Gorenstein 
ring and $H^1_*(\E)$ is generated by 
a single nonzero element.
\subsection{The moduli space $\mathcal{B}(0,c_2)$ and its Known components}
We define $\mathcal{P}(\boldsymbol{a}; \boldsymbol{b})$ to be the set of minimal Horrocks monads as in \eqref{eq:hoorcks-monad} and we can see that
$$\mathcal{P}(\boldsymbol{a}; \boldsymbol{b})=\{\beta:\cA^\vee\rightarrow\cB \mbox{ locally left invertible }: \beta^\vee\Omega\beta=0\},$$ 
where $\Omega$ is a fixed symplectic form on $\cB$. Let $\mathcal{V}(\boldsymbol{a}; \boldsymbol{b})$ denote the set of isomorphism 
classes of stable rank 2 bundles on $\mathbb{P}^3$ with even determinant which are given as cohomology 
of a monad in $\mathcal{P}(\boldsymbol{a}; \boldsymbol{b})$. Consider 
$G=\{\varphi\in\End(\cB):\varphi^\vee\circ\Omega\circ\varphi=\varphi\}$. We define an action on 
$\mathcal{P}(\boldsymbol{a}; \boldsymbol{b})$ by
$$(u,\varphi)\cdot\beta=\varphi^{-1}\beta u^{-1}, \forall (u,\varphi)\in\gl(\cA)\times G.$$
For $G_0=\gl(\cA)\times G/\pm(\id,\id)$, the above action acts freely on $\mathcal{P}(\boldsymbol{a}; \boldsymbol{b})$ 
and two monads in $\mathcal{P}(\boldsymbol{a}; \boldsymbol{b})$ are isomorphic if, and 
only if, they lie in the same $G_0$-orbit. In this manner, we obtain a well-defined map
$$\mathcal{P}(\boldsymbol{a}; \boldsymbol{b})\longrightarrow\mathcal{V}(\boldsymbol{a}; \boldsymbol{b})$$
which is an isomorphism, see \cite[Section 8]{MF2021}. Therefore, there is an injective morfism
$$\mathcal{V}(\boldsymbol{a}; \boldsymbol{b})\longrightarrow\mathcal{B}(0,c_2),$$
where $c_2$ is expressed in terms of $\boldsymbol{a}$ and $\boldsymbol{b}$ as in formula \eqref{eq:monad-c2}.

As stated in the Introduction, there are three known irreducible components of $\mathcal{B}(0,c_2)$: the \textit{Instanton component} $I_{c_2}$ whose generic 
point is a stable rank 2 bundle $\E$ on $\mathbb{P}^3$ given as cohomology of a minimal monad $\mathbf{M}$ in $\mathcal{P}(\boldsymbol{a}; \boldsymbol{b})$ with $\boldsymbol{a}=(1^{c_2})$ and $\boldsymbol{b}=(0^{2c_2+2})$, i.e. $\mathbf{M}$ has the form  
$$0\longrightarrow\mathcal{O}_{\mathbb{P}^3}(-1)^{\oplus c_2}\stackrel{\alpha}{\longrightarrow}{\mathcal{O}_{\mathbb{P}^3}}^{\oplus (2c_2+2)}\stackrel{\beta}{\longrightarrow}\mathcal{O}_{\mathbb{P}^3}(1)^{\oplus c_2}\longrightarrow0.$$
It was shown in \cite{H78} that $\dim I_{c_2}=8c_2-3$. The other known irreducible component of $\mathcal{B}(0,c_2)$ is the \textit{Ein component} 
$\mathcal{G}(r,s,t)$ with $r, s,t$ integers 
such that $c_2=t^2-r^2-s^2$, and whose generic point is a stable rank 2 bundle $\mathcal{F}$ arising as 
cohomology of an Ein minimal monad as defined in Example \ref{ein-monad}. If we expand Equation 2.2.B in \cite{Ein88} which 
gives the dimension of $\mathcal{G}(r,s,t)$ we have
\begin{eqnarray}\label{ein-dim}
\dim \mathcal{G}(r,s,t)=&\binom{t+r+3}{3}+\binom{t+s+3}{3}+\binom{t-r+3}{3}+\binom{t-s+3}{3}-\binom{r+s+3}{3}\nonumber\\
 &-\binom{s-r+3}{3}-\binom{2r+3}{3}-\binom{2s+3}{3}-3-\mu(r,s),
\end{eqnarray}
where $\mu(r,s)=4$ if $r=s=0$ or $\mu(r,s)=1$ if $r=0<s$ or $r=s>0$ or $\mu(r,s)=0$ if $0<r\leq s$.

Finally, we have the so-called \textit{modified Instanton component} $\mathcal{M}_{c_2}$. Let $a, m$ be integers the condition that 
$$\mbox{ neither } 5\leq a\leq11, 1+\epsilon\leq m+\epsilon\leq a-4 \mbox{ or } a\geq12, 1+\epsilon\leq m+\epsilon\leq a+1$$
with $c_2=2m+\epsilon+a^2, \epsilon\in\{0,1\}$. According to \cite[Theorem 1]{tikhomirov2019}, there exist a family of irreducible components $\mathcal{M}_{c_2}$ of $\mathcal{B}(0,c_2)$ of dimension 
\begin{eqnarray*}
\dim \mathcal{M}_{2m+\epsilon+a^2}=4\binom{a+3}{3}+(2m+\epsilon)(10-a)-11.
\end{eqnarray*}
\subsection{Serre Correspondence}
In addition to obtaining rank 2 vector bundles on $\mathbb{P}^3$ as cohomology of a minimal 
monad (as in Subsection \ref{sub-monads}), we can also construct rank 2 vector bundles via the Hartshorne-Serre 
correspondence. More precisely for any fixed invertible sheaf $\mathcal{L}=\OO_{\mathbb{P}^3}(c)$, let $(Y,\xi)$ be a 
pair where $Y$ is a local complete intersection Cohen-Macaulay curve on $\mathbb{P}^3$ and $\xi\in H^0(\omega_Y(4-c))$ is a 
global section which generates the sheaf $\omega_Y(4-c)$, except finitely many points.  Using 
the isomorphism $H^0(\omega_Y(4-c))\simeq\Ext^1(\mathcal{I}_Y,\mathcal{O}_{\mathbb{P}^3}(-c))$,
the global section $\xi$ induces the exact sequence 
\begin{equation}\label{seq-serre}
 0\rightarrow\mathcal{O}_{\mathbb{P}^3}(-c)\stackrel{\cdot s}{\rightarrow}\mathcal{E}\rightarrow\mathcal{I}_Y\rightarrow0,   
\end{equation}
so that $\mathcal{E}$ is a rank 2 bundle on $\mathbb{P}^3$ with 
$c_1(\mathcal{E})=c$ and $c_2(\mathcal{E})=\deg Y.$
Conversely, if $\mathcal{E}$ is a rank 2 vector bundle on $\mathbb{P}^3$ with $c_1(\mathcal{E})=c_1$ and 
$s\in H^0(\E(k))$ is a global section whose zero-set has codimension 2, then one obtains an exact 
sequence as in \eqref{seq-serre} and a pair $(Y,\xi)$ where $Y=(s)_0$ and 
$\xi\in\Ext^1(\mathcal{I}_Y,\mathcal{O}_{\mathbb{P}^3}(-c_1))$, see \cite[Theorem 4.1]{H80} for more details.\\


\section{Computation of $\Ext^1$ Groups}\label{section-ext}
\indent Only a small number of results assist in determining the dimension of the group $\Ext^1$. Among 
these, \cite[Lemma 4.1.7]{Okonek} is one of the most familiar, though it demands the vanishing of various cohomology 
groups, thereby limiting its range of application. The goal of this section is to compute the dimension 
of $\Ext^1(\E,\E)$ and $\Ext^2(\E,\E)$ for a stable rank 2 bundle $\E\in\mathcal{B}(0,c_2)$ on 
$\mathbb{P}^3$ arising as the cohomology of a minimal Horrocks monad as in \eqref{eq:hoorcks-monad}, where the maps 
$\alpha$ and $\beta$ are known explicitly.  As a consequence, in Subsection \ref{family2}, we provide examples 
of new irreducible components of $\mathcal{B}(0,c_2)$ that contain singular points.\\
Let $\E\in\mathcal{B}(0,c_2)$ be a stable rank 2 bundle on $\mathbb{P}^3$. From sympletic structure of $\E$ it follows that $\End(\E,\E)\simeq \E\otimes\E$ and from deformation Theory, see for example \cite{Sernesi}, we know that 
$$\dim T_{[\E]}\mathcal{B}(0,c_2)=\dim\Ext^1(\E,\E).$$
Consequently $\dim\mathcal{B}(0,c_2)\leq\dim\Ext^1(\E,\E),$ where the equality is true if, and only if, $\E$ is 
a smooth point which is equivalent to the vanishing $h^2(\End(\E,\E))=0$. Since $c_1(\E)=0$, Proposition 4.2 of \cite{H78} yields
$$\dim \Ext^1(\E,\E)-\dim\Ext^2(\E,\E)=8c_2-3.$$
For a minimal monad as in \eqref{eq:hoorcks-monad} whose cohomology is a stable rank 2 bundle $\E$, let $A$ and $B$ be the 
matrices of the maps $\alpha$ and $\beta$, respectively. We provide the following outline of 
Macaulay2 code for computing the dimension of $\Ext^1(\E,\E)$: 
\begin{lstlisting}[language=Macaulay2]
S = QQ[x,y,z,w], X=Proj S;
-- Define the source and target modules for alpha and beta
H = (S(b_1)++ S(-b_1))++..++(S(b_{s+1}).. S(-b_{s+1}))
L1= S^{-a_1}++S^{-a_2}++..++S^{-a_s}; L2 = S^{a_1}++S^{a_2}++..++S^{a_s};
Define the map alpha and beta
beta = map(L2, H, B); alpha=map(H, L1,A);
-- Construct the vector bundle E
F= ker sheaf beta;I=image sheaf alpha; E=F/I
-- Computing the dimension of the tangent space on E
T=E**E;
HH^1 T
\end{lstlisting}
To illustrate the method, we work through an explicit example.
\begin{example}\label{ex1}
\textup{Let $\E\in\mathcal{B}(0,9)$ be a rank 2 bundle on $\mathbb{P}^3$ given as cohomology of the minimal monad
\begin{equation*}
0\rightarrow2\cdot\mathcal{O}_{\mathbb{P}^{3}}(-3)\stackrel{\alpha}{\rightarrow}2\cdot\mathcal{O}_{\mathbb{P}^{3}}(2)\oplus2\cdot\mathcal{O}_{\mathbb{P}^{3}}(-2)\oplus \mathcal{O}_{\mathbb{P}^{3}}(1)\oplus\mathcal{O}_{\mathbb{P}^{3}}(-1)\stackrel{\beta}{\rightarrow}2\cdot\mathcal{O}_{\mathbb{P}^{3}}(3)\rightarrow0,
\end{equation*}
where 
$$B=\left(
\begin{array}{cccccc}
x& 0&w^{5}& z^{5}& y^{2}& z^{4}\\
y& x& z^{5}& w^{5}& x^{3}& 0\\
\end{array} \right),
~~A=\left(
\begin{array}{cc}
w^{5}&0\\
z^{5}&w^{5}+xz^{4}\\
-x&0\\
-y&-x\\
0&-z^{4}\\
yz&xz+y^{2}\\
\end{array}
\right)$$
are the matrices of $\beta$ and $\alpha$ respectively. The code below, executed in Macaulay2, shows that $\dim\Ext^1(\E,\E)=79$.}
\begin{lstlisting}[language=Macaulay2]
S = QQ[x,y,z,w],X=Proj S;
H = S^{2}^2 ++ S^{{-2}}^2 ++ S^{1} ++ S^{(-1)};
L1= S^{-3}^2, L2 = S^{3}^2;

beta = map(L2, H, {
{x, 0, w^(5), z^(5), y^(2), z^(4)},
{y, x, z^(5), w^(5), x^(2), 0}})

alpha=map(H, L1,{{w^(5), 0},{z^(5), w^(5) + x*z^(4)},
{-x, 0},{-y, -x},{0, -z^(4)},{y*z, x*z + y^(2)}})
-- Construct the vector bundle E
F= ker sheaf beta; I=image sheaf alpha; E=F/I
-- Computing the dimension of the tangent space on E
T=E**E; HH^1 T
\end{lstlisting}

\end{example}

\begin{remark}
This method has two limitations: we must know the matrices of the maps $\alpha$ and $\beta$; the computation of $\Ext^1$ can 
be extensive, depending on the entries of these matrices.
\end{remark}

The moduli space $\mathcal{B}(0,c_2)$ is completely classified for $c_2\leq5$, so the next case of interest 
is $c_2=6$. The moduli space $\mathcal{B}(0,6)$ has only two known irreducible components: the \textit{instanton 
component} whose generic point is a stable rank 2 bundle given as cohomology of a minimal monad 
with $\boldsymbol{a}=(1^6), \boldsymbol{b}=(0^{14})$ and the \textit{modified Ein component} $\mathcal{M}_6$ whose 
generic point is a stable rank 2 bundle given as cohomology of a minimal monad with 
$\boldsymbol{a}=(2,1^2), \boldsymbol{b}=(0^{8})$.

Using the same Macaulay2 code, we present in Table \ref{ext-tab} the $\dim T_{[\E]}\mathcal{B}(0,6)=e$ 
for a stable bundle $\E$ arising as the cohomology of a minimal monad corresponding to case 6(3) or 6(4) in 
Table 5.3 of Hartshorne and Rao \cite{HR91}. These dimensions contribute to the classification of 
the moduli space $\mathcal{B}(0,6)$, see Remark \ref{rmk1} for example.

\begin{table}[h]
\centering
\begin{tabular}{|l|l|l|l|l|}
\hline
\rowcolor[HTML]{EFEFEF} 
\multicolumn{1}{|c|}{\cellcolor[HTML]{EFEFEF}$\boldsymbol{a}$}& \multicolumn{1}{|c|}{\cellcolor[HTML]{EFEFEF}$\boldsymbol{b}$}& \multicolumn{1}{c|}{\cellcolor[HTML]{EFEFEF}$e$} & \multicolumn{1}{c|}{\cellcolor[HTML]{EFEFEF}$\alpha$} & \multicolumn{1}{c|}{\cellcolor[HTML]{EFEFEF}$\beta$} \\ \hline
$0^2,1^2$                                       & $2^2$   &45&  
$\begin{pmatrix}
z^{3}&x^{2}w\\
x^{3}&z^{3}\\
0&-y\,w+w^{2}\\
0&-x^{2}\\
-w&-y\\
-y&0
\end{pmatrix}         $                                           &   $\begin{pmatrix}
y&0&x^{2}&w^{2}&0&z^{3}\\
w&y&0&w^{2}&z^{3}&x^{3}
\end{pmatrix}$                                                   \\ \hline
$0^4,2$& $3,1$ &48       &$ \begin{pmatrix}
w^{5} & x^{3} \\
0 & -w \\
0 & z \\
x^{3} & 0 \\
-z^{3} & -y \\
-y & 0
\end{pmatrix}$
                                                 &   $\begin{pmatrix}
y & 0 & 0 & z^3 & x^3 & w^{5} \\
0 & z & w & y & 0 & x^{3}
\end{pmatrix}$                                                   \\ \hline
                               $0^2,1^3$ & $2^2,1$&45 &    $\begin{pmatrix}
0 & -w^3 & y^2 \\
x^3 & w^3 & -yz \\
z^3 & x^3 & -yw + zw \\
0 & 0 & -z + w \\
0 & 0 & x \\
-w & -y & 0 \\
-y & 0 & 0 \\
w & z & 0
\end{pmatrix}$                                                   &     $\left(\!\begin{array}{ccc}
z&w&0\\
y&w&0\\
0&y&0\\
0&0&x\\
0&0&z-w\\
w^{3}&x^{3}&y\,z\\
x^{3}&z^{3}&y\,w-z\,w\\
w^{3}&0&y^{2}
\end{array}\!\right)^{T}$                 \\ \hline
\end{tabular}
\caption{Computation of the $\Ext^1(\E,\E)$ for some stable bundles with $c_1(\E)=0, c_2(\E)=6$.}
\label{ext-tab}
\end{table}
The moduli space $\mathcal{B}(-1,6)$ has only three known irreducible components: the Instanton component 
of dimension 43 and two Ein components of dimension 43 and 50, respectively. Furthermore, 
\cite[Theorem 18]{MF2021} ensures the existence of a new component $M_4$ although its dimension has not 
been computed. The component $\mathrm{M}_4$ contains the stable rank 2 bundles on $\mathbb{P}^3$ that are given as cohomology of minimal 
monads as in \eqref{monad-c1=-1}. Recall that for any $\E\in\mathcal{B}(-1,6)$ we have 
$\E^\vee(-1)\cong \E$ and consequently $\operatorname{End}(\E)\cong \E(1)\otimes\E$.
\begin{theorem}\label{ex2}
Let $\mathrm{M}_4$ be the irreducible component of $\mathcal{B}(-1,6)$ containing stable rank 2 bundles on $\mathbb{P}^3$ 
given as cohomology of minimal monads of the form
\begin{equation}\label{monad-c1=-1}
0\rightarrow2\cdot\mathcal{O}_{\mathbb{P}^{3}}(-3)\stackrel{\alpha}{\rightarrow}3\cdot\mathcal{O}_{\mathbb{P}^{3}}(1)\oplus3\cdot\mathcal{O}_{\mathbb{P}^{3}}(-2)\stackrel{\beta}{\rightarrow}2\cdot\mathcal{O}_{\mathbb{P}^{3}}(2)\rightarrow0.
\end{equation}
Then $\dim \mathrm{M}_4=45$.
\end{theorem}

\begin{proof}
From Theorem~18 of \cite{MF2021}, we know that $\dim \mathrm{M}_4\geq45$. Conversely, for $\E\in \mathrm{M}_4$ given as cohomology of 
a minimal monad as in \eqref{monad-c1=-1} where the maps 
$\beta$ and $\alpha$ are represented, respectively, by the matrices
$$B=\begin{pmatrix}
  y & x & z & 0 & w^{4} & x^{4} \\
  w & z & y & x^{4} & 0 & w^{4}
\end{pmatrix},
~~A=\begin{pmatrix}
  0 & x^{4} \\
  w^{4} & 0 \\
  x^{4} & w^{4} \\
  -y & -w \\
  -x & -z \\
  -z & -y
\end{pmatrix},$$
we employ the following Macaulay2 code to compute $\dim\operatorname{Ext}^1(\E,\E)$:
\begin{lstlisting}[language=Macaulay2]
S = QQ[x,y,z,w];X=Proj S;
H = S^{1}^3++ S^{-2}^3
L1=S^{-3}^2
L2=S^{2}^2
B=map(L2,H,matrix{{y,x,z,0,w^4,x^4},{w,z,y,x^4,0,w^4}})
A=map(H,L1,matrix{{0,x^4},{w^4,0},{x^4,w^4},{-y,-w},{-x,-z},{-z,-y}})
E=(ker sheaf B)/(image sheaf A)
T=(E(1))**E;
HH^1 T
\end{lstlisting}
The output of the above code shows that $\dim\Ext^1(\E,\E)=45$. Since the tangent space at $\E$ satisfies 
$\dim\mathrm{T}_{\E}\mathrm{M}_4 = \dim\Ext^1(\E,\E)=45$, 
it follows that $\dim\mathrm{M}_4\leq45$. Therefore, the dimension of $\mathrm{M}_4$ is exactly $45$ and that $\E$ is a 
smooth point of $\mathrm{M}_4$.
\end{proof}

\section{Some families}\label{section-families}
In this section, we establish the existence of two infinite families of irreducible components, $\mathbf{V_0}$ and 
$\mathbf{V}_1$, of the moduli space $\mathcal{B}(0,c_2).$
 \subsection{A family of rank 2 bundles $\mathcal{V}(\boldsymbol{a};\boldsymbol{b})$ with $\boldsymbol{a}=({(a-1)}^2,0^2)$ and $\boldsymbol{b}=(a^2)$}\label{family1}
Let $a>2$ be an integer. If we consider the maps $\alpha$ and $\beta$ such that
 $$\alpha=\begin{pmatrix}
     z^{2a-1}&x^aw^{a-1}\\
     x^{2a-1}& z^{2a-1}\\
     0& -yw^{a-1}+w^a\\ 
     0&-x^a\\
     -w&-y\\
     -y&0
     \end{pmatrix}\ \ \ \ \mbox{ and } \ \ \ \beta=\begin{pmatrix}
         y & 0 & x^a & w^a & 0 & z^{2a-1}\\
         w&y&0&w^a&z^{2a-1}&x^{2a-1}
     \end{pmatrix},$$
then we can observe that $\alpha$ is injective, $\beta$ is surjective and $\beta\circ\alpha=0$. This means that the family of complexes  
\begin{equation}\label{monad0}
\mathbf{M_0}:0\rightarrow2\cdot\mathcal{O}_{\mathbb{P}^{3}}(-a)\stackrel{\alpha}{\rightarrow}2\cdot\mathcal{O}_{\mathbb{P}^{3}}(a-1)\oplus2\cdot\mathcal{O}_{\mathbb{P}^{3}}\oplus 2\cdot\mathcal{O}_{\mathbb{P}^{3}}(1-a)\stackrel{\beta}{\rightarrow}2\cdot\mathcal{O}_{\mathbb{P}^{3}}(a)\rightarrow0,
\end{equation}
provides a family of minimal Horrocks monads. For a fixed integer $a\geq2$, the cohomology of the associated  minimal monad is a 
stable rank 2 bundle $\mathcal{E}$ on $\mathbb{P}^3$ with Chern 
classes $c_1(\mathcal{E})=0, c_2(\mathcal{E})=4a-2$. Moreover, from the display to the minimal monad associated, we obtain
$$h^1(\mathcal{E}(-l))=2h^0(\mathcal{O}_{\mathbb{P}^{3}}(a-l))-2h^0(\mathcal{O}_{\mathbb{P}^{3}}(a-l-1))=\left\{\begin{array}{cc}
  0   & l\geq a+1 \\
  (a-l+1)(a-l+2)   & 1\leq l\leq a. 
\end{array}
\right.$$
Denoting by $\mathcal{X}$ the spectrum of $\mathcal{E}$, it follows from \cite[Property S4]{HR91} and the above formula that
$$\#\{k_j\in\mathcal{X}:k_j\geq i-1\}=h^1(\mathcal{E}(-i))-h^1(\mathcal{E}(-i-1))=2, i=1,\cdots, a,$$
which means
$$\mathcal{X}=\{{(1-a)}^2,\cdots,-1^2,0^2,1^2,\cdots, {(a-1)}^2\}.$$
On the other hand, since the family of minimal monads in \eqref{monad0} is homotopy free, we can apply 
\cite[formula 33]{MF2021} to compute its dimension. In this case, we have
\begin{equation*}
h=16+4\binom{a+3}{3}+4\binom{2a+2}{3}; w=\binom{2a+3}{3}, g=4; s=7+4\binom{a+2}{3}+3\binom{2a+1}{3}.
\end{equation*}
If we consider $\boldsymbol{a}=(a^2)$ and $\boldsymbol{b}=({(a-1)}^2,0^2)$, then we proved that
\begin{equation}\label{dim-f}
\dim\mathcal{V}(\boldsymbol{a}; \boldsymbol{b})=6a^2+6a+8.    
\end{equation}
 
\begin{remark}\label{rmk1}
From \eqref{dim-f} we observe that 
 $$\dim\mathcal{V}(\boldsymbol{a};\boldsymbol{b})>8c_2-3=32a-19, \mbox{ whenever } a>2,$$
 implying that the family $\mathcal{V}(\boldsymbol{a};\boldsymbol{b})$ cannot contained in the Instanton component 
 of $\mathcal{B}(0,4a-2)$. For $a=2$ the family of minimal Horrocks monads
\begin{equation}
0\rightarrow2\cdot\mathcal{O}_{\mathbb{P}^{3}}(-2)\stackrel{\alpha}{\rightarrow}2\cdot\mathcal{O}_{\mathbb{P}^{3}}(1)\oplus2\cdot\mathcal{O}_{\mathbb{P}^{3}}\oplus 2\cdot\mathcal{O}_{\mathbb{P}^{3}}(-1)\stackrel{\beta}{\rightarrow}2\cdot\mathcal{O}_{\mathbb{P}^{3}}(2)\rightarrow0
\end{equation}
coincides with \cite[Table 5.3, $c_2=6$, (3)]{HR91}. This family of minimal monads has dimension $44$ which is smaller than the expected dimension 45. This indicates that the family of bundles given as cohomology of minimal monads of these minimal monads does not fill an irreducible component of $\mathcal{B}(0,6)$. 
\end{remark}

For $a=3$, there is no Ein component $\mathcal{G}(r,s,t)$ because the only possible pairs of $(r,s)$ are $(r,s)=(1,1)$ or $(r,s)=(1,3)$ which would require $t^2=8$ or $t^2=20$, respectively-neither of which yields an integer 
$t$ satisfying the required conditions. The other particular case is $a=4$ for which the only solution is 
$(r,s,t)=(1,1,4)$. Here the Ein irreducible component has dimension $117$ 
whereas $\mathcal{V}(4;1^2)$ has dimension $128$.

\begin{proposition}\label{prop-dim1}
Let $t, r, s$ be positive integers such that $t>r+s$. If $\mathcal{G}(t; r,s)$ is an Ein irreducible component of the 
moduli space $\mathcal{B}(0,4a-2)$, then 
$$\dim\mathcal{V}(\boldsymbol{a}; \boldsymbol{b})>\dim \mathcal{G}(t; r,s)$$
and consequently, the 
family $\mathcal{V}(\boldsymbol{a}; \boldsymbol{b})$ is not contained in any Ein irreducible component.
\end{proposition}
\begin{proof}
We have $c_2=t^2-r^2-s^2=4a-2$ and if $s$ is an even integer (or $r$ is an even integer) then $s=2k$ for some integer $k$ and thus $(t-r)(t+r)=2(2a+2k^2-1)$ implying $2t$ an odd integer, which is a contradiction. Hence, $r,s$ are necessarily odd integers. Furthermore, the condition $t>r+s$ implies $rs<2a-1$. All information concerning $t, r, s$ and $a$ is presented below.
\begin{itemize}
    \item $r, s$ are odd integers while $t$ is an even integer;
    \item $rs<2a-1$;
\item $t\geq r+s+2$;
    \item $t<a$ for all $a>4$.
\end{itemize}

From Equation \eqref{ein-dim}, we have
\begin{eqnarray*}
 \dim \mathcal{G}(t; r,s)=&-\frac{4}{3}r^{3}-r^{2}s-\frac{5}{3}s^{3}+r^{2}t+s^{2}t+\frac{2}{3}t^{3}-4\,r^{2}-4\,s^{2}\\
 &+4\,t^{2}-\frac{11}{3}r-\frac{22}{3}s+\frac{22}{3}t-3-\epsilon, \mbox{ where } \epsilon \in\{0,1\}.   
\end{eqnarray*}
Since $c_2=t^2-r^2-s^2=4a-2$, we obtain $a=(t^2-r^2-s^2+2)/4$ and
$$\begin{matrix}
\dim \mathcal{V}(\boldsymbol{a},\boldsymbol{b})= & \frac{3}{8}t^{4}+\frac{3}{4}r^{2}s^{2}+\frac{3}{8}s^{4}-\frac{3}{4}r^{2}t^{2}-\frac{3}{4}s^{2}t^{2}+\frac{3}{8}t^{4}-3r^{2}-3s^{2}+3t^{2}+\frac{25}{2}.
\end{matrix}$$
 If we define the funcion $h(r,s,t)=\dim \mathcal{V}(\boldsymbol{a},\boldsymbol{b})-\dim \mathcal{G}(r,s,t)$, then $h(r,s,t)$ can be rewritten
 
 \begin{eqnarray*}
 h(r,s,t)=&\frac{3}{8}r^{4}+\frac{3}{4}r^{2}s^{2}+\frac{3}{8}s^{4}-\frac{3}{4}r^{2}t^{2}-\frac{3}{4}s^{2}t^{2}+\frac{3}{8}t^{4}
 +\frac{4}{3}r^{3}+r^{2}s\\
 &+\frac{5}{3}s^{3}-r^{2}t-s^{2}t-\frac{2}{3}t^{3}+r^{2}+s^{2}-t^{2}+\frac{11}{3}r+\frac{22}{3}s-\frac{22}{3}t+\frac{31}{2}+\epsilon.   
 \end{eqnarray*}
For fixed integers $r_0, s_0$ with $s_0\geq r_0\geq 1$, a simple calculus argument shows that $h(r_0,s_0,t)>0$ for all $t\geq r_0+s_0+2$. Indeed, let $h_t(t)$ denote the derivative of $h(r_0,s_0,t)$ with respective to $t$, hence
$$h_t(t)=-\frac{3}{2}{r_0}^{2}t-\frac{3}{2}{s_0}^{2}t+\frac{3}{2}t^{3}-{r_0}^{2}-{s_0}^{2}-2t^{2}-2t-\frac{22}{3}.$$
and 
 $$h_{tt}(t)=-\frac{3}{2}{r_0}^{2}-\frac{3}{2}{s_0}^{2}+\frac{9}{2}t^{2}-4t-2=\frac{3}{2}(-{r_0}^{2}-{s_0}^{2}+t^2)+\frac{6}{2}t^{2}-4t-2=\frac{3}{2}c_2+\frac{6}{2}t^{2}-4t-2,$$
where $c_2=t^2-{r_0}^{2}-{s_0}^{2}$. Since $t\geq r_0+s_0+2\geq 4$, we see that $h_{tt}(t)>0$ whenever $t\geq r_0+s_0+2$. This shows that  $h_t(t)$ is increasing on
$[r_0+s_0+2,\infty)$. Furthermore,
$$h_t(r_0+s_0+2)=3{r_0}^{2}s_0+3{r_0}{s_0}^{2}+3{r_0}^{2}+14r_0s_0+3{s_0}^{2}+8r_0+8s_0-\frac{22}{3}>0,$$ 
which implies $h_t(t)>0$ on this interval. Consequently, $h(t)$ is increasing on  $[r_0+s_0+2,\infty)$ with 
$$\begin{matrix}
    h(r_0+s_0+2)&=&\frac{3}{2}{r_0}^{2}{s_0}^{2}-\frac{1}{3}{r_0}^{3}+4{r_0}^{2}{s_0}+3{r_0}{s_0}^{2}+8{r_0}{s_0}-\frac{11}{3}r_0-\frac{5}{2}+\epsilon\\
    &=&\frac{3}{2}{r_0}^{2}{s_0}^{2}+{r_0}^2(-\frac{1}{3}{r_0}+4{s_0})+3{r_0}{s_0}^{2}+r_0(8{s_0}-\frac{11}{3})-\frac{5}{2}+\epsilon.
\end{matrix}$$
Since we are assuming $t_0>s_0\geq r_0$, it follows that $-\frac{1}{3}{r_0}+4{s_0}>0$ and thus $h(r_0+s_0+2)>0$. By the monotonicity of $h$ we conclude  $h(t)>0$ for all $t \in [r_0+s_0+2,\infty)$.
\end{proof}
\begin{proposition}\label{prop-dim2}
The family of stable rank 2 bundles $\mathcal{V}(\boldsymbol{a};\boldsymbol{b})$, with $\boldsymbol{a}=(a^2)$ and $\boldsymbol{b}=({(a-1)}^2,0^2)$, has dimension larger than the dimension of the modified Instanton component $\mathcal{M}_{c_2}$ of the moduli space $\mathcal{B}(0,c_2)$.
 \end{proposition}
 \begin{proof}
There are integers $u, m$ and $\epsilon\in\{0,1\}$ such that $c_2=2m+\epsilon+u^2=4a-2$, where either 
$u\leq4$ or $5\leq u\leq11, 1+\epsilon\leq m+\epsilon\leq u-4$ or $u\geq12, 1+\epsilon\leq m+\epsilon\leq u+1$. For the integers $(m, \epsilon,u)$ such that $c_2=2m+\epsilon+u^2$ with $c_2\leq18$, in \cite[Subsection 4.1]{tikhomirov2019} was showed that the dimension of $\mathcal{M}_{c_2}$ coincides with the expected dimension. From Remark \ref{rmk1}, it follows that $\mathcal{V}(\boldsymbol{a};\boldsymbol{b})$ cannot 
be contained in $\mathcal{M}_{c_2}$ for $u\leq4$. \\
We now examine the two cases that remain. We first observe that if $u\geq a$, 
then $2m+\epsilon+u^2=4a-2\leq 4u-2$ implying $u^2-4u+2+2m+\epsilon\leq0$. This inequality can only hold for 
$u\leq4$. Therefore, $u<a$ whenever $u\geq5$. According to \cite[Equation 53]{tikhomirov2019},
\begin{eqnarray*}
\dim \mathcal{M}_{2m+\epsilon+u^2}=4\binom{u+3}{3}+(2m+\epsilon)(10-u)-11.
\end{eqnarray*}
Dividing $\dim \mathcal{M}_{2m+\epsilon+u^2}$ by $4a-2$ yields
\begin{eqnarray*}
\dim \mathcal{M}_{2m+\epsilon+u^2}=(4a-2)(\frac{2}{3}u+4)+(-m)(\frac{10}{3}u-12)-\frac{5}{3}u\epsilon+\frac{22}{3}u+6\epsilon+4.
\end{eqnarray*}

Since $\frac{10}{3}u-12\geq0$ for $u\geq4$ and $-m\leq-1$ we have 
$$-m(\frac{10}{3}u-12)\leq-1(\frac{10}{3}u-12)=-\frac{10}{3}u+12.$$
Now, for $5\leq u<a$, it follows that
\begin{eqnarray*}
\dim \mathcal{M}_{2m+\epsilon+u^2}&\leq&(4a-2)(\frac{2}{3}u+4)-\frac{10}{3}u+12-\frac{5}{3}u\epsilon+\frac{22}{3}u+6\epsilon+4\\
&<&(4a-2)(\frac{2}{3}a+4)-\frac{10}{3}a+12-\frac{5}{3}a\epsilon+\frac{22}{3}a+6\epsilon+4\\
&=& \frac{8}{3}a^2-\frac{5}{3}a\epsilon+\frac{56}{3}a+6\epsilon+8.
\end{eqnarray*}
The quadratic function
\begin{eqnarray*}
h(a)=\dim\mathcal{V}(\boldsymbol{a},\boldsymbol{b})-(\frac{8}{3}a^2-\frac{5}{3}a\epsilon+\frac{56}{3}a+6\epsilon+8)=\frac{10}{3}a^2-\frac{38}{3}a+\frac{5}{3}a\epsilon-6\epsilon,
\end{eqnarray*}
is positive for $a\geq4$. Therefore, $\dim \mathcal{V}(\boldsymbol{a},\boldsymbol{b})>\mathcal{M}_{2m+\epsilon+u^2}$ for $a\geq4$, as desired.

\end{proof}
\begin{theorem}\label{dim-family11}
Let $a\geq3$ be an integer. There exists a new component $\mathcal{Y}_a\in\mathbf{V}_0$ of the moduli space $\mathcal{B}(0,4a-2)$ distinct 
from the instanton, Ein, and modified Instanton components, whose dimension is at least $6a^2+6a+8$.
\end{theorem}
\begin{proof}
From Remark \ref{rmk1} together with Propositions \ref{prop-dim1} and \ref{prop-dim2} it follows that 
$\mathcal{V}(\boldsymbol{a};\boldsymbol{b})$, with $\boldsymbol{a}=(a^2)$ and $\boldsymbol{b}=({(a-1)}^2,0^2)$, has 
dimension large than that of the instanton, Ein and modified Instanton components. Hence 
$\mathcal{V}(\boldsymbol{a};\boldsymbol{b})$ must lie contained in a new component $\mathcal{Y}_a$ 
of $\mathcal{B}(0,4a-2)$ and consequently,
$$\dim \mathcal{Y}_a\geq 6a^2+6a+8.$$
\end{proof}

\subsection{A family of rank 2 bundles $\mathcal{V}(\boldsymbol{a};\boldsymbol{b})$ with $\boldsymbol{a}=({(a-1)}^2,1)$ and $\boldsymbol{b}=(a^2)$}\label{family2}
In this subsection, we will consider an integer $a\geq3$ and the minimal Horrocks monad family 
\begin{equation}\label{monad-1}
\mathbf{M}_1: 2\cdot\op3(-a)\stackrel{\alpha}\rightarrow2\cdot\op3(a-1)\oplus2\cdot\op3(1-a)\oplus(\op3(1)\oplus\op3(-1))\stackrel{\beta}\rightarrow2\cdot\op3(a),
\end{equation}
where 
$$\beta=\left(
\begin{array}{cccccc}
x& 0&w^{2a-1}& z^{2a-1}& y^{a-1}& z^{a+1}\\
y& x& z^{2a-1}& w^{2a-1}& x^{a-1}& 0\\
\end{array} \right), 
~~\alpha=\left(
\begin{array}{cc}
w^{2a-1}&0\\
z^{2a-1}&w^{2a-1}+x^{a-2}z^{a+1}\\
-x&0\\
-y&-x\\
0&-z^{a+1}\\
yz^{a-2}&xz^{a-2}+y^{a-1}\\
\end{array}
\right)$$
such that $\alpha$ is injective, $\beta$ is surjective and $\beta\alpha=0$. According to our notation, we set $\mathbf{M}_1:=\mathcal{P}({(a-1)}^2,1;a^2)=\mathcal{P}(\boldsymbol{a};\boldsymbol{b})$, where $\boldsymbol{a}=({(a-1)}^2,1)$ and $\boldsymbol{b}=(a^2)$. From \cite[Lemma 13]{F2024}, if $\E$ is the cohomology of a minimal monad of the family $\mathcal{P}(\boldsymbol{a};\boldsymbol{b})$, then $\E$ is a stable rank 2 bundle with $c_1(\E)=0, c_2(\E)=4a-3$ and its spectrum is given by
\begin{equation}\label{spectrum1}
 \mathcal{X}:=\mathcal{X}(\E)=\{{(1-a)}^2,\cdots,{-1}^2,0,1^2,\cdots,{(a-2)}^2, {(a-1)}^2\}.   
\end{equation}
It was also proved in \cite{F2024} that $\dim\mathcal{V}(\boldsymbol{a};\boldsymbol{b})=6a^2+6a+2$ which is larger 
than the expected dimension of the moduli space $\mathcal{B}(0,4a-3)$. Following the same argument as in Proposition 
\ref{prop-dim2} for the case $c_2=4a-3$ we obtain 
\begin{equation}\label{dim-4}
\dim\mathcal{V}(\boldsymbol{a};\boldsymbol{b})>\dim\mathcal{M}_{c-2} \mbox{ for } a\geq3    
\end{equation}
and consequently, the family $\mathcal{V}(\boldsymbol{a};\boldsymbol{b})$ cannot be contained in the modified 
Instanton component $\mathcal{M}_{c-2}$.\\
\indent The goal of this subsection is to 
show that the family $\mathcal{V}(\boldsymbol{a};\boldsymbol{b})$ is not contained the Ein and modified irreducible components. For this, it remains to verify that 
$\mathcal{V}(\boldsymbol{a};\boldsymbol{b})$ cannot contained in the Ein irreducible components.  
\begin{proposition}\label{prop-dim3}
For $\boldsymbol{a}=(a^2)$ and 
$\boldsymbol{b}=({(a-1)}^2,1)$, the family of stable rank 2 bundles $\mathcal{V}(\boldsymbol{a};\boldsymbol{b})$ 
cannot be contained in any irreducible Ein component $\mathcal{G}(r,s,t)$ of the moduli 
space $\mathcal{B}(0,4a-3)$ for $a\geq3$.
 \end{proposition}
\begin{proof}
The case $a=3$ was showed in \cite[Theorem 14]{F2024}. Let $\mathcal{G}(r,s,t)$ be an Ein irreducible component 
associated to a triple of integers $r,s, t$ such that $t>r+s$ and $c_2=t^2-r^2-s^2=4a-3$. We will 
separeted the proof in two cases: $t=r+s+1$ and $t\geq r+s+2$.\\
\textbf{Case 1:} For \( t = r + s + 1 \), the following cases arise:
\begin{enumerate}
    \item[a)] \((r,s,t) = (0,s,s+1)\);
    \item[b)] \((r,s,t) = (1,s,s+2)\);
    \item[c)] \((r,s,t) = (r,s,r+s+1)\) with \( r > 1 \).
\end{enumerate}
We can observe that \( a = \frac{1}{2}(rs + r + s + 2) \). Moreover, from Equation \eqref{ein-dim}
\[
\dim \mathcal{G}(r,s,r+s+1) = -\frac{1}{3}r^{3} + 2r^{2}s + 3rs^{2} + 3r^{2} + 12rs + 3s^{2} + \frac{41}{3}r + 10s + 9 - \mu(r,s),
\]
where \(\mu(r,s) \in\{0,1,4\}\), and
\[
\dim \mathcal{V}(\boldsymbol{a};\boldsymbol{b}) = \frac{3}{2}r^{2}s^{2} + \frac{3}{2}r^{2} + \frac{3}{2}s^{2} + 3r^{2}s + 3rs^{2} + 12rs + 9r + 9s + 14.
\]
Consequently,
\begin{eqnarray*}
h(r,s) & := & \dim \mathcal{V}(\boldsymbol{a};\boldsymbol{b}) - \dim \mathcal{G}(r,s,r+s+1) \\
& = & \frac{3}{2}r^{2}s^{2} - \frac{1}{3}r^{3} + r^{2}s - \frac{3}{2}r^{2} - \frac{3}{2}s^{2} - \frac{14}{3}r - s + 5 + \mu(r,s) .
\end{eqnarray*}
For the case \((r,s,t) = (1,s,s+2)\) it follows that $a = \frac{1}{2}(2s+3)$, which is not an integer.
Now, we assume that $(r,s,t)=(0,s,s+1)$. Then
\[
h(0,s) = -\frac{3}{2}s^{2} - s + 5 + \mu(0,s),
\]
which is negative for all sufficiently large \( s \). Observe that the identity \((s+1)^2 - s^2 = 4a - 3\) yields \( s = 2a - 2 \). Since \( a > 2 \), it follows that \( t = s + 1 > s = 2a - 2 > a \).  
Now, if \( \mathcal{F} \) denotes a generic point of \( \mathcal{G}(0,s,s+1) \), then \( h^1(\mathcal{F}(-s-1)) = 1 \). By 
lower semicontinuity, \( h^1(\mathcal{E}(-s-1)) \geq 1 \) for every 
\( \mathcal{E} \in \mathcal{G}(0,s,s+1) \). However, for any \( \mathcal{E}' \in \mathcal{V}(\boldsymbol{a};\boldsymbol{b}) \) we have \( h^1(\mathcal{E}'(-s-1)) = 0 \). Hence, 
\( \mathcal{V}(\boldsymbol{a};\boldsymbol{b}) \) cannot be contained 
in \( \mathcal{G}(0,s,s+1) \).

Finally, consider \( r > 1 \). We may write
\[
h(r,s) = (r^2 - 1)\left( \frac{3}{2}s^{2} - \frac{1}{3}r + s - \frac{3}{2} \right) - 5r + \frac{7}{2} + \mu(r,s) .
\]
Define \( g(r) = \dfrac{-5r + \frac{7}{2}}{r^2 - 1} \). we verify that \( |g(r)| < 1 \) for all \( r > 4 \), and moreover
\[
g(2) = -\frac{13}{6}, \quad g(3) = -\frac{23}{16}, \quad g(4) = -\frac{11}{10}.
\]

Hence,
\[
\begin{aligned}
h(r,s) & = (r^2 - 1)\left( \frac{3}{2}s^{2} - \frac{1}{3}r + s - \frac{3}{2} + g(r) \right) + \mu(r,s) \\
& \geq (r^2 - 1)\left( \frac{3}{2}s^{2} - \frac{1}{3}r + s - \frac{3}{2} - \frac{13}{6} \right) + \mu(r,s) \\
& = (r^2 - 1)\left( \frac{3}{2}s^{2} + s - \frac{1}{3}r - \frac{11}{3} \right) + \mu(r,s) .
\end{aligned}
\]
Since \( s \geq r > 1 \), it follows that \( r^2 > 1 \), \( s > \frac{1}{3}r \), and \( \frac{3}{2}s^2 > \frac{11}{3} \). Therefore, 
the expression inside the parentheses is positive, which implies \( h(r,s) > 0 \).

\textbf{Case 2:} Consider \( t \geq r + s + 2 \).

Define \( h(r,s,t) := \dim \mathcal{V}(\boldsymbol{a};\boldsymbol{b}) - \dim \mathcal{G}(r,s,t) \).  
First, assume \( r = s = 0 \). Then
\[
h(0,0,t) = \frac{3}{8}t^4 - \frac{2}{3}t^3 - \frac{1}{4}t^2 - \frac{22}{3}t + \frac{135}{8}.
\]
The polynomial \( h(0,0,t) \) does not admit real roots; consequently, \( h(0,0,t) > 0 \) for all  
\( t>0 \). Finally, we suppose \( s > 0 \). Since \( t \geq r + s + 2 \),  we can adapt the method used 
in the proof of Proposition \ref{prop-dim1} to establish the positivity \( h(r,s,t) > 0 \).
\end{proof}

\begin{corollary}\label{family-2}
Let $a\geq3$ be an integer, $\boldsymbol{a}=(a^2)$ and 
$\boldsymbol{b}=({(a-1)}^2,1)$. There exists a new component $\mathcal{Z}_a\in\mathbf{V}_1$ of the moduli space 
$\mathcal{B}(0,4a-3)$ distinct 
from the instanton, Ein, and modified Instanton components, whose dimension is at least $6a^2+6a+2$.
\end{corollary}
\begin{proof}
From Equation \eqref{dim-4}, it follows that 
$\mathcal{V}(\boldsymbol{a};\boldsymbol{b})$, with $\boldsymbol{a}=(a^2)$ and $\boldsymbol{b}=({(a-1)}^2,1^2)$, has 
dimension larger than that of the instanton and modified Instanton components, and from Proposition \ref{prop-dim3}, cannot be contained in any irreducible Ein component. Hence 
$\mathcal{V}(\boldsymbol{a};\boldsymbol{b})$ must lie contained in a new component $\mathcal{Y}_a$ 
of $\mathcal{B}(0,4a-3)$ and consequently,
$$\dim \mathcal{Y}_a\geq 6a^2+6a+2.$$
\end{proof}
\begin{theorem}
Under the hypothesis of Corollary \ref{family-2} and Proposition \ref{prop-dim3}, the new component $\mathcal{Z}_a$ of 
$\mathcal{B}(0,4a-3)$ contains singular points.
\end{theorem}
\begin{proof}
In addition to the maps $\alpha$ and $\beta$ defined at the beginning of this subsection \ref{family2}, consider
$$\tilde{\alpha}= \begin{pmatrix}
y^{2a-1} & z^{2a-1} \\[4pt]
z^{2a-1} + y^{a-1} w^{a} & y^{2a-1} \\[4pt]
-y^{a} w - x w^{a} & -x y^{a} \\[4pt]
-y^{a-1} & 0 \\[4pt]
-x & -w \\[4pt]
-w & -x
\end{pmatrix}\ \ \ \ \mbox{ and } \ \ \ \ \tilde{\beta} = \begin{pmatrix}
w & x & y^{a-1} & 0 & z^{2a-1} & 0 \\
x & w & 0 & w^{a+1} & y^{2a-1} & z^{2a-1}
\end{pmatrix}$$
where $\tilde{\alpha}$ is injective, $\tilde{\beta}$ is surjective and $\tilde{\beta}\tilde{\alpha}=0$. This yields other minimal 
monad in $\mathcal{P}(\boldsymbol{a};\boldsymbol{b})$ whose cohomology $\tilde{\mathcal{E}}:=\ker\tilde{\beta}/\im\tilde{\alpha}$ 
is a stable rank 2 bundle on $\mathbb{P}^3$ belonging to 
$\mathcal{V}(\boldsymbol{a};\boldsymbol{b})$. We also set $\mathcal{E}:=\ker\beta/\im\alpha$ and denote 
$$\tilde{e} := \dim \operatorname{Ext}^1(\tilde{\mathcal{E}}, \tilde{\mathcal{E}}) \mbox{ and } e := \dim \operatorname{Ext}^1(\mathcal{E}, \mathcal{E}).$$
Using the code from Section \ref{section-ext} and comparing with Example \ref{ex1} for $a=3$, we obtain $(\tilde{e}, e) = (76, 79)$, while for $a > 3$ we find $e - \tilde{e} = 4$. Moreover, we observe that 
$\dim \mathcal{V}(a,b) + 2a = e$ holds for any $a > 3$. This can be verified using the code below.

 \begin{lstlisting}[language=Macaulay2]
a=3--any a>2
S = QQ[x,y,z,w], X=Proj S;
H = S^{a-1}^2 ++ S^{1}++S^{-1} ++ S^{1-a}^2;
L1= S^{-a}^2, L2= dual L1;
A1=map(H,L1,matrix {
{y^(2*a-1), z^(2*a-1)},
{z^(2*a-1) + y^(a-1)*w^a, y^(2*a-1)},
{-y^a*w - x*w^a, -x*y^a},{-y^(a-1), 0},{-x, -w},{-w, -x}})
B1=map(L2,H,matrix {
{w, x, y^(a-1), 0, z^(2*a-1), 0},
{x, w, 0, w^(a+1), y^(2*a-1), z^(2*a-1)}})

A2=map(H, L1,{
{w^(2*a-1), 0},
{z^(2*a-1), w^(2*a-1) + x^(a-2)*z^(a+1)},
{0, -z^(a+1)},
{y*z^(a-2), x*z^(a-2) + y^(a-1)},{-x, 0},{-y, -x}})
B2 = map(L2, H, {
{x, 0,  y^(a-1), z^(a+1),w^(2*a-1), z^(2*a-1)},
{y, x, x^(a-1), 0, z^(2*a-1), w^(2*a-1)}})

E1=(ker sheaf B1)/(image sheaf A1)
HH^0 E1--Stability condition
V1=E1**E1; HH^1 V1--tangent

E2=(ker sheaf B2)/(image sheaf A2)
HH^0 E2--Stability condition
V2=E2**E2; HH^1 V2--tangent
--dim
N1=Ext^1(module(V1),S)
N2=Ext^1(module(V2),S)
hilbertFunction(0,N1)
hilbertFunction(0,N2)
6*a^2+8*a+2
\end{lstlisting}
\end{proof}

\subsection{Generalizations}
We have already shown the existence of two families with $c_2=4a-2-b$, where $b\in\{0,1\}$. Moreover, these families belong to a new irreducible component of $\mathcal{B}(0,c_2)$ for every $a>2$. In particular, the smallest Chern class for which they define a new component is $c_2=9$. We observe that case $8(4)-i)$ in \cite[Table 5.3]{HR91} shares similarities with the two families we constructed; this motivates the following conjecture.
\begin{conjecture}\label{conjecture}
 Let \(a\) and \(k\) be positive integers. For each integer \(b\in\{0,1,\dots,k-1\}\) with \(a>b\), there exists a stable rank-2 vector bundle \(\mathcal{E}\) arising as the cohomology of a minimal monad of the form
\begin{equation}\label{general-monad}
    \cM(a, b,k) : 0 \lra {\mathcal{O}_{\mathbb{P}^{3}}(-a)}^{\oplus k} \stackrel{\alpha}{\lra} \begin{matrix} {\mathcal{O}_{\mathbb{P}^{3}}(a-1)}^{\oplus k} \oplus \mathcal{O}_{\mathbb{P}^{3}}(b) \\ \displaystyle \oplus \\ {\mathcal{O}_{\mathbb{P}^{3}}(1-a)}^{\oplus k}\oplus\mathcal{O}_{\mathbb{P}^{3}}(-b) 
\end{matrix} \stackrel{\beta}{\lra} {\mathcal{O}_{\mathbb{P}^{3}}(a)}^{\oplus k} \lra 0.
\end{equation}
Moreover, for all sufficiently large integers \(a \gg 0\), the corresponding family of such bundles belongs to an irreducible component of  \(\mathcal{B}(0,k(2a-1)-b^2)\), whose dimension exceeds the expected dimension. 
\end{conjecture}

We observe that for $k=1$ the Conjecture \ref{conjecture} reduces to a case of Ein's monads. We have seen 
in Subsections \ref{family1} and \ref{family2} that for \(k=2\), the values \(b=0\) or \(b=1\) satisfy this 
condition for all \(a>2\). More precisely, we have shown that for 
\(b \in \{0,1\}\), the family of minimal monads \(\cM(a,b,2)\) yields stable rank 2 bundles with spectrum 
\(\{0^{2-b}, 1^2, 2^2, \dots, (a-1)^{2}\}\). Moreover, we proved in Theorem \ref{dim-family11} and Corollary 
\ref{family-2} that if \(a>2\), then this family of bundles lies in a new component 
of \(\mathcal{B}(0,4a-2-b)\), distinct from the 
Instanton, Ein and modified Instanton components.

In sequence, we study some particular cases to $k=3$ and $k=4$. For $k=3$, if we apply \cite[Lemma 4.8]{HR91} to 
the monads $\mathbf{M}_0$ of \eqref{monad0} and $\mathbf{M}_1$ of \eqref{monad-1} considering $r=a$ and a complete intersection curve $X$ of 
degree $(2a-1,1)$ on $\mathbb{P}^3$, there are stable bundles $\E_0$ and $\E_1$, respectively, such that 
$c_1(\E_0)=c_1(\E_1)=0,c_2(\E_0)=6a-2$ and $c_2(\E_1)=6a-3$. Moreover, $\mathcal{X}(\E_b)=\{0^{3-b}, 1^3, 2^3, \dots, (a-1)^{3}\}\) and the 
stable bundle $\E_b$ is given as 
cohomology of a minimal monad of the form $\cM(a, b,3)$ with $b=0,1$. However, this approach does not cover all possibilities; the case $b=2$ still requires evaluation.

\textbf{Case $k=3$}. Let \[\alpha=
\begin{pmatrix}
  0 & x^{\eta} - w^{\eta} & 0 \\
  x^{\eta} - z^{\eta} & y w^{\eta-1} - w^{\eta} & -z^{\eta} + w^{\eta} \\
  0 & w^{\eta} & x^{\eta} - w^{\eta} \\
  x^{a+b} & 0 & 0 \\
  -y^{a-b} + z^{a-b} & 0 & 0 \\
  w & z & y + w \\
  -w & -y & -w \\
  0 & -z & -y
\end{pmatrix}
,\] and \[\beta=\begin{pmatrix}
  z & 0 & y & 0 & 0 & w^{\eta} & w^{\eta} & x^{\eta} \\
  y & w & w & 0 & 0 & z^{\eta} & x^{\eta} & z^{\eta} \\
  -z & -w & -y-w & y^{a-b} - z^{a-b} & x^{a+b} & x^{\eta} - z^{\eta} - 2w^{\eta} & -2w^{\eta} & -z^{\eta} - w^{\eta}
\end{pmatrix}\]
be the maps of monad $\cM(a,b,3)$, with $\eta=2a-1$. 
A direct computation shows that $\beta \circ \alpha = 0$. Moreover, the map $\beta$ admits no syzygies of degree zero, and one verifies that both $\alpha$ and $\beta$ have maximal rank at every point of $\mathbb{P}^3$; in particular, $\alpha_p$  is injective ($\beta_p$ is surjective) for all $p \in \mathbb{P}^3$. Since the matrices are matrices of homogeneous forms, and they
contain no non-zero scalar entries, then  $\alpha$ and $\beta$ are both minimal. This means that $\E=\ker \beta/\rm im \ \alpha$ is a stable rank $2$ bundle.  

Denote by $\mathcal{V}(a,b,3):=\mathcal{V}({a}^k;({a-1})^k,{b})$ the corresponding family of stable bundles of $\cM(a,b,3)$.
Applying formula (33) of \cite{MF2021} yields, for all $a>3$:
\begin{itemize}
     \item$\dim \mathcal{V}(a,2,3)=9a^{2}+6a-3$;
     \item$\dim \mathcal{V}(a,1,3)=9a^2+6a+13$; 
     \item$\dim \mathcal{V}(a,0,3)=9a^{2}+6a+18$.
\end{itemize}
 For all \(a>3\) and \(b\in\{0,1,2\}\), we obtain the inequality  $\dim\mathcal{V}(a,b,3) > 8c_2-3$,
where \(c_2 = 6a - 3 - b^2\). This verifies Conjecture~\ref{conjecture} for \(k=3\). Furthermore, by applying the 
methods of Sections~4.1 and~4.2, we conclude that this family of bundles belongs to a new irreducible 
component of \(\mathcal{B}(0,c_2)\), distinct from the three known components (instanton, Ein, or modified instanton). In contrast, the dimensions of $\mathcal{V}(2,b,3)$ are smaller than expected. We need to analyze $a=3$.
 
 For $a=3$ and $b=0$, we have $\dim\mathcal{V}(3,0,3)=117=8c_2-3$, where $c_2=15$.
Let $\E\in\mathcal{V}(3,0,3)$ be the bundle defined by the maps $\alpha$ and $\beta$.
Using the algorithm described in Section~\ref{section-ext}, we compute $h^{1}(\E(-2))=9$; hence $\E$ has odd Atiyah--Rees invariant and therefore cannot belong to the instanton component. For $c_2=15$, the only possible Ein monads correspond to the triples $(r,s,t)\in\{(0,1,4),\;(0,7,8),\;(1,3,5)\}$,
so in every case $t>3$. By semicontinuity, $\E$ cannot belong to an Ein component. Moreover, there are no modified instantons with $c_2=15$. Consequently, $\mathcal{V}(3,0,3)$ lies in a new irreducible component of the moduli space. 

A similar analysis applies to $\mathcal{V}(3,1,3)$. It cannot belong to any of the three known components. 
Furthermore, it cannot lie in the same component as $\mathcal{V}(4,0,2)$ because these two families have distinct Atiyah--Rees invariants. Thus, $\mathcal{V}(3,1,3)$ also lies in a new component.  We list the special case $a=3$ in Table \ref{table:2}.
    \begin{table}[h]
 \centering
\begin{tabular}{|l|c|c|c|c|}
\hline
$c_2$ & $\cM(a,b,3)$ &  \mbox{ expected dimension } & $\dim \mathcal{V}(a,b,3)$ & $\chi(\E)$\\ \hline
         15    & $(3,0,3)$     &  $117$    &      $117$&$\{0^3,1^3,2^3\}$  \\ \hline
    14   &$(3,1,3)$ &       $109$     &   $112$  &$\{0^2,1^3,2^3\}$\\ \hline
      11   & $(3,2,3)$    &   $85$         &      93   & $\{0,1^2,2^3\}$ \\ \hline
\end{tabular}
\caption{Dimension of $\mathcal{V}(a,b,3)$ for $a=3$.}
\label{table:2}
\end{table}

\textbf{Case $k=4$}. For $a=3$, the dimensions of $\mathcal{V}(a,b,4)$ corresponding to $b=0,1,2$ are $151$, $147$, and $130$, respectively. This indicates that the dimension exceeds the expected dimension only when $b=2$. Moreover, we have $\dim\mathcal{V}(4,0,4)=239$, $\dim\mathcal{V}(4,1,4)=235$, $\dim\mathcal{V}(4,2,4)=222$, and $\dim\mathcal{V}(4,3,4)=189$; hence, for all $b\in\{0,1,2,3\}$, the inequality $\dim\mathcal{V}(3,b,4) > 8c_2-3$ holds. For $a>4$, the dimension of $\mathcal{V}(a,b,4)$ is larger than expected and is given by:
\begin{itemize}
    \item $\dim \mathcal{V}(a,0,4)=12a^{2}+4a+31$,
    \item $\dim \mathcal{V}(a,1,4)=12a^{2}+4a+27$,
    \item $\dim \mathcal{V}(a,2,4)=12a^{2}+4a+14$,
    \item $\dim \mathcal{V}(a,3,4)=12a^{2}+4a-15$.
\end{itemize}
Consequently, Conjecture~\ref{conjecture} is verified for $k=4$. We now present the explicit matrices $B$ and $A$ representing the maps $\beta$ and $\alpha$, respectively.
$$ B=\begin{pmatrix}
   z & w & 0 & 0 & 0 & 0 & 0 & 0 & x^{2a-1} & y^{2a-1} \\
    0 & z & w & 0 & 0 & 0 & 0 & x^{2a-1} & y^{2a-1} & x^{2a-1} \\
    w & 0 & z & x & 0 & 0 & z^{2a-1} & y^{2a-1} & 0 & y^{2a-1}+z^{2a-1} \\
    0 & z & w & x & y^{a-b}-z^{a-b} & w^{a+b} & z^{2a-1} & x^{2a-1} & y^{2a-1} & z^{2a-1}
\end{pmatrix}$$
$$A= \begin{pmatrix}
   0 & 0 & x^{2a-1} & y^{2a-1} \\
   0 & x^{2a-1} & y^{2a-1} & 0 \\
 0 & y^{2a-1} & 0 & x^{2a-1} \\
    z^{2a-1} & 0 & -x^{2a-2}w & -x^{2a-2}z \\
    w^{a+b} & w^{a+b} & 0 & 0 \\
    -y^{a-b}+z^{a-b} & -y^{a-b}+z^{a-b} & 0 & 0 \\
    -x & 0 & w & z \\
    0 & -z & w & z-w \\
    0 & -w & -z & 0 \\
    0 & 0 & -w & -z
\end{pmatrix}$$ 

We have thus established the validity of the Conjecture ~\ref{conjecture} for $k<5$ and all $a>k$. The cases with $a \leq k$ still require analysis, since for some of them we have $\dim \mathcal{V}(a,b,k) > 8c_2-3$. The family of bundles arising from $\mathcal{V}(a,b,k)$ exhibits gaps in the possible values of $c_2$; for instance, there are no stable bundles on $\mathcal{V}(a,b,k)$ with $c_2=12$, although semi-stable bundles do occur.


\bibliographystyle{amsalpha}
\bibliography{biblio}
\end{document}